\documentclass[12pt,final,1p,times,review]{elsarticle}

\usepackage{amssymb}
\usepackage{amsthm}
\usepackage{hyperref}
\usepackage{amsmath}
\usepackage{amsmath,amssymb,amsopn,amsfonts,mathrsfs,amsbsy,amscd}
\usepackage{longtable}
\usepackage{newtxtext, newtxmath}
\usepackage{geometry}
\journal{???}

\newcommand{\e}{\epsilon}

\newtheorem{definition}{Definition}[section]
\newtheorem{theorem}{Theorem}[section]

\newtheorem{lemma}{Lemma}[section]

\newtheorem{remark}{Remark}[section]
\numberwithin{equation}{section}

\begin{document}
\bibliographystyle{amsplain}
\begin{frontmatter}
\title{\textbf{The Classification of the Rota-Baxter Operators on Four-dimensional Grassmann Algebra}}

\author[label1]{Kamran Shakoor\corref{mycorrespondingauthor}}
\cortext[mycorrespondingauthor]{Kamran Shakoor}
\author[label1]{Noor-ul-Ain}
\address[label1]{Abdus Salam School of Mathematical Sciences,\\ GC University, Lahore-Pakistan\\e-mail:
			kamranshakoor@sms.edu.pk, noor-ul-ain-25@sms.edu.pk}

\begin{abstract}
The purpose of this paper is to classify all the Rota-Baxter operators explicitly on Grassmann algebra in dimension four over the reals.
\end{abstract}

\Huge

\end{frontmatter}

\section{Introduction} \label{Introduction}

Rota--Baxter algebras were first formulated by Glen E. Baxter \cite{Baxter1960} in 1960,
motivated by problems in probability theory, specifically fluctuation theory for
random walks. The identity Baxter isolated,
\[
R(x)R(y) = R\big(R(x)y + xR(y)\big),
\]
was subsequently recognized by Rota \cite{Rota1969, Rota1995} as an algebraic structure of
independent combinatorial interest, and Rota's foundational work — together with that
of his collaborators and students — established Rota--Baxter algebras as a central
object in modern algebraic combinatorics. In recent decades, the Rota--Baxter
algebraic framework has proven instrumental in resolving analytical and combinatorial
problems \cite{Tang2}, and has found broad applications across diverse fields of mathematics and
mathematical physics. Perhaps the most celebrated of these is the connection, discovered independently by Belavin
and Drinfeld \cite{BelavinDrinfeld} and by Semenov-Tian-Shansky \cite{SemenovTianShansky}, between Rota--Baxter operators of weight zero on Lie algebras and solutions of the classical Yang--Baxter equation \cite{Gao} — an equation that itself plays a fundamental role in symplectic geometry, integrable systems, quantum groups, and topological quantum field theory. Rota--Baxter structure also underlies the splitting of associative algebras into dendriform \cite{EG-book} and tridendriform \cite{BaiGuoNi} structures, governs the algebraic
Birkhoff decomposition \cite{Guo} used in Connes--Kreimer renormalization \cite{ConnesKreimer} of perturbative quantum field theories, and appears in the theory of shuffle algebras and multiple zeta values, in operads and Hopf algebras, and in the integration and geometrization of Lie groups and Lie algebroids (see \cite{Chu, Tang1} and the references therein for surveys of these directions).

Classifying Rota--Baxter operators on particular algebras stands as a central concern
of the theory, since an explicit classification both illuminates the operator's role
as a noncommutative analogue of integration by parts and provides the raw data needed
to build the associated dendriform and pre-Lie structures on the algebra. Significant
progress has been made for low-dimensional cases: Rota--Baxter operators on $2$- and
$3$-dimensional algebras have been thoroughly investigated in \cite{An, Li, Guo2}. For the second-order full matrix algebra over the complex field, with weight $0$, explicit classifications were derived using both standard computational methods and Gr\"obner basis techniques in \cite{Tang2}, and
the classification of Rota--Baxter operators on $\mathfrak{sl}_2(\mathbb{C})$ has
likewise received considerable attention \cite{Gubarev}. For unital algebras more
generally, it has been shown that every Rota--Baxter operator of nonzero weight on
the Grassmann algebra over a field of characteristic zero is a \emph{splitting
operator} — that is, a projection onto a subalgebra along a complementary subalgebra \cite{Gubarev, Rota1969}. More recently, Rota--Baxter
operators of arbitrary weight on the split semiquaternion algebra have been fully
classified by Chen and Deng \cite{chen}; see also \cite{Gubarev} for a broader survey
of classification techniques for Rota--Baxter operators on low-dimensional and
matrix-type algebras. These results motivate the natural next step of classifying
Rota--Baxter operators, of arbitrary weight, on the remaining $4$-dimensional
quaternion-type algebras — an attempt we pursue here for the degenerate case of generalized quaternion
algebra, which is the \textit{Grassmann} algebra $\bigwedge(\mathbb{R}^2)$. The results here for $\bigwedge(\mathbb{R}^2)$ serve as an independent classification of intrinsic interest for a well-known
$4$-dimensional nilpotent-graded algebra. Our results thus contribute to the growing body of literature on operator structures on low-dimensional algebras, and situate $\bigwedge(\mathbb{R}^2)$ as a natural companion case to existing classifications on quaternion-type algebras.

While our classification of the degenerate case provides a crucial baseline for the general theory on generalized quaternion algebras, it also reveals the far greater intricacy of the non-degenerate setting, clarifying why that full classification remains a considerably more sophisticated and computationally demanding undertaking.

\section{Preliminaries}

\subsection{The Grassmann Algebra \texorpdfstring{$\bigwedge(\mathbb{R}^2)$}{Λ(R2)}}

The Grassmann algebra, also known as the \emph{exterior algebra}, is a fundamental algebraic structure that encodes the notion of oriented area and volume in a coordinate-free manner. It was introduced by Hermann Grassmann $(1809-1877)$ in his $1844$ treatise \emph{Die lineale Ausdehnungslehre} \cite{Grassmann1844} and has since become an essential tool in multilinear algebra, differential geometry, and mathematical physics; see e.g.\ \cite{Bourbaki, Marcus, Sternberg} for modern treatments.

Let $V = \mathbb{R}^2$ be the real vector space with standard basis $\{e_1, e_2\}$.
The exterior algebra $\bigwedge(V) = \bigwedge(\mathbb{R}^2)$ is the quotient of the
tensor algebra $T(V) = \bigoplus_{k \ge 0} V^{\otimes k}$ by the two-sided ideal $I$
generated by all elements of the form $v \otimes v$ for $v \in V$:
\[
\bigwedge(V) \;=\; T(V)/I .
\]
The induced product, denoted $\wedge$, is called the \emph{wedge product}. The
defining relation $v \wedge v = 0$ for all $v \in V$, together with bilinearity,
forces the \emph{anticommutativity}
\[
v \wedge w = -\,w \wedge v \qquad \text{for all } v, w \in V.
\]
Since $\dim V = 2$, the algebra decomposes into homogeneous components
\[
\bigwedge(\mathbb{R}^2) \;=\; \bigwedge{}^0(V) \oplus \bigwedge{}^1(V) \oplus \bigwedge{}^2(V),
\qquad
\bigwedge{}^0(V) \cong \mathbb{R}, \quad
\bigwedge{}^1(V) \cong V, \quad
\bigwedge{}^2(V) \cong \mathbb{R},
\]
so that $\dim_{\mathbb{R}} \bigwedge(\mathbb{R}^2) = 2^2 = 4$. As a basis for this
$4$-dimensional algebra we take
\[
\{\, e_0 := 1,\; e_1,\; e_2,\; e_{3} := e_1 \wedge e_2 \,\}.
\]
The multiplication is determined by
\[
e_1^2 = 0, \qquad e_2^2 = 0, \qquad e_1 \wedge e_2 = -\,e_2 \wedge e_1 = e_{3},
\]
with $1$ acting as the multiplicative identity, and it consequently satisfies
$e_i \wedge e_{3} = e_{3} \wedge e_i = 0$ for $i=1,2$ and $e_{3}\wedge e_{3}=0$.
Explicitly, the multiplication table of $\bigwedge(\mathbb{R}^2)$ with respect to
the basis $\{e_0, e_1, e_2, e_{12}\}$ is
\[
\begin{array}{c|cccc}
\wedge & e_0 & e_1 & e_2 & e_{3} \\ \hline
e_0    & e_0 & e_1 & e_2 & e_{3} \\
e_1    & e_1 & 0   & e_{3} & 0 \\
e_2    & e_2 & -e_{3} & 0 & 0 \\
e_{3} & e_{3} & 0 & 0 & 0
\end{array}
\]

Thus $\bigwedge(\mathbb{R}^2)$ is a unital, associative, $\mathbb{Z}_2$-graded
(super)commutative algebra that is neither commutative nor a division algebra: it is
graded as $\bigwedge(\mathbb{R}^2) = A_{\bar 0} \oplus A_{\bar 1}$ with
$A_{\bar 0} = \operatorname{span}\{e_0, e_{3}\}$ and
$A_{\bar 1} = \operatorname{span}\{e_1, e_2\}$, and elements of odd degree
anticommute while elements of even degree commute. Every element of
$A_{\bar 1} \oplus \operatorname{span}\{e_{3}\}$ is nilpotent; $e_{3}$ is the
highest-degree (volume) element and squares to zero. The element $e_{3}$ spans the
socle of the algebra and generates a $1$-dimensional ideal contained in the center
\[
Z\big(\bigwedge(\mathbb{R}^2)\big) = \operatorname{span}\{e_0, e_{3}\}.
\]
The addition for arbitrary two elements is given by
\[
\sum_{i=0}^{3}x_ie_i+\sum_{i=0}^{3}y_ie_i
=
\sum_{i=0}^{3}(x_i+y_i)e_i.
\]

The multiplication for arbitrary two elements is given by
\begin{align*}
	&(x_0e_0+x_1e_1+x_2e_2+x_3e_3)
	(y_0e_0+y_1e_1+y_2e_2+y_3e_3)\\
    &=(x_0y_0)e_0+(x_0y_1+x_1y_0)e_1+(x_0y_2+x_2y_0)e_2+(x_0y_3+x_1y_2-x_2y_1+x_3y_0)e_3.
\end{align*}
One can easily check that $\bigwedge(\mathbb{R}^2)$ with the above addition and multiplication is an \textit{associative} algebra. We call $\bigwedge(\mathbb{R}^2)$ the \textit{exterior algebra} or \textit{Grassmann algebra} of a vector space $V$. For a complete exposition of these technical details, we defer to the definitive treatises of \cite{Bourbaki, Marcus}.

\begin{remark}[Degeneration of generalized quaternion algebras]
The algebra $\bigwedge(\mathbb{R}^2)$ arises naturally as a \emph{degenerate case}
of the generalized quaternion algebras. Recall that, for $a, b \in F^{\times}$, the
generalized quaternion algebra $\left(\dfrac{a,b}{F}\right)$ over a field $F$ is the
associative unital algebra generated by $i, j$ subject to
\[
i^2 = a, \qquad j^2 = b, \qquad ij = -ji =: k
\]
(see \cite{Lam, Voight} for the general theory). Formally setting $a = b = 0$
collapses this family precisely to $\bigwedge(F^2)$, via the identification
$i \mapsto e_1$, $j \mapsto e_2$, $k \mapsto e_{3}$; over $F = \mathbb{R}$ this gives
$\bigwedge(\mathbb{R}^2) \cong \left(\dfrac{0,0}{\mathbb{R}}\right)$. In this sense
$\bigwedge(\mathbb{R}^2)$ is the ``boundary'' member of the quaternion family: it
retains the $4$-dimensional structure and the anticommutation relation $ij=-ji$, but
loses the division-algebra property, since $i, j, k$ become nilpotent rather than
square roots of nonzero scalars. This perspective places the study of Rota--Baxter
operators on $\bigwedge(\mathbb{R}^2)$ within the broader context of operator theory
on quaternion-type algebras.
\end{remark}

\subsection{Rota--Baxter Operators}

\begin{definition} (\cite{GuoKellerYau})
Let $A$ be a (not necessarily associative) algebra over a field $\mathbb{K}$ and
let $\lambda \in \mathbb{K}$. A linear map $\mathcal{P} : A \to A$ is called a
\emph{Rota--Baxter operator of weight $\lambda$} if
\[
\mathcal{P}(x) \mathcal{P}(y) = \mathcal{P}\big( \mathcal{P}(x) y + x \mathcal{P}(y) + \lambda\, xy \big) \qquad \text{for all } x, y \in A.
\]
When $\lambda = 0$, $\mathcal{P}$ is called a Rota--Baxter operator of \emph{weight zero}. Thus, the pair $(A,\mathcal{P})$ constitutes a \emph{Rota--Baxter algebra of weight $\lambda$}.
\end{definition}

We adopt the following conventions: $\mathbb{R}$ is the field of real numbers; all linear maps are over $\mathbb{R}$; and given a matrix $M$, $M^T$ denotes the transpose of $M$, and
$$
M(4) =
\begin{pmatrix}
M & 0 & 0 & 0 \\
0 & M & 0 & 0 \\
0 & 0 & M & 0 \\
0 & 0 & 0 & M
\end{pmatrix}
.
$$

\section{The Rota-Baxter operators on \texorpdfstring{$\bigwedge(\mathbb{R}^2)$}{Λ(R2)}}
Let $\mathcal{P}$ be a linear transformation on \texorpdfstring{$\bigwedge(\mathbb{R}^2)$}{Λ(R2)}, and
\[
P = \begin{pmatrix}
p_{11} & p_{12} & p_{13} & p_{14} \\
p_{21} & p_{22} & p_{23} & p_{24} \\
p_{31} & p_{32} & p_{33} & p_{34} \\
p_{41} & p_{42} & p_{43} & p_{44}
\end{pmatrix}
= (\eta_0, \eta_1, \eta_2, \eta_3),
\]
the matrix of $\mathcal{P}$ with respect to the basis $\{e_0, e_1, e_2, e_3\}$. Notice easily that $\eta_i$ is the coordinate of $\mathcal{P}$ with respect to the basis $\{e_0, e_1, e_2, e_3\}$. Thus we have
\[
\mathcal{P}(e_i) = (e_0, e_1, e_2, e_3)\eta_i = \eta_i^T
\begin{pmatrix}
e_0 \\
e_1 \\
e_2 \\
e_3
\end{pmatrix}.
\]

\bigskip

\begin{lemma}\label{lem1}
Let $\mathcal{P}$ be a linear transformation on \texorpdfstring{$\bigwedge(\mathbb{R}^2)$}{Λ(R2)}. The following equation holds:
\begin{equation}\label{eq}
    \mathcal{P}(e_i)\mathcal{P}(e_j) = \eta_i^T C \eta_j(4)
    \begin{pmatrix}
    e_0 \\
    e_1 \\
    e_2 \\
    e_3
    \end{pmatrix},
\end{equation}
where
\[
C= \left(
\begin{array} {cccccccccccccccc}
	1 & 0 & 0 & 0 & 0 & 1 & 0 & 0 & 0 & 0 & 1 & 0 & 0 & 0 & 0 & 1\\
	0 & 0 & 0 & 0 & 1 & 0 &  0 & 0 & 0 & 0 & 0 & 0 & 0 & 0 & 1 & 0\\
	0 & 0 & 0 & 0 & 0 & 0 & 0 & 0 & 1 & 0 & 0 & 0 & 0 & -1 & 0 & 0\\
	0 & 0 & 0 & 0 & 0 & 0 & 0 & 0 & 0 & 0 & 0 & 0 & 1 & 0 & 0 & 0
	\end {array}
	\right).
	\]
\end{lemma}
\begin{proof}
For any $i$ and $j$, since
\[
\mathcal{P}(e_i)\mathcal{P}(e_j) = \eta_i^T
\begin{pmatrix}
e_0 \\
e_1 \\
e_2 \\
e_3
\end{pmatrix}
(e_0, e_1, e_2, e_3)\eta_j
= \eta_i^T
\begin{pmatrix}
e_0 & e_1 & e_2 & e_3 \\
e_1 & 0 &  e_3 & 0  \\
e_2 & -e_3 & 0 & 0 \\
e_3 & 0 & 0  & 0
\end{pmatrix}
\eta_j
\]
\[
= \eta_i^T
\begin{pmatrix}
1 & 0 & 0 & 0 \\
0 & 0 & 0 & 0 \\
0 & 0 & 0 & 0 \\
0 & 0 & 0 & 0
\end{pmatrix}
\eta_j e_0 + \eta_i^T
\begin{pmatrix}
0 & 1 & 0 & 0 \\
1 & 0 & 0 & 0 \\
0 & 0 & 0 & 0 \\
0 & 0 & 0  & 0
\end{pmatrix}
\eta_j e_1
\]
\[
+ \eta_i^T
\begin{pmatrix}
0 & 0 & 1 & 0 \\
0 & 0 & 0 & 0 \\
1 & 0 & 0 & 0 \\
0& 0 & 0 & 0
\end{pmatrix}
\eta_j e_2 + \eta_i^T
\begin{pmatrix}
0 & 0 & 0 & 1 \\
0 & 0 & 1 & 0 \\
0 & -1 & 0 & 0 \\
1 & 0 & 0 & 0
\end{pmatrix}
\eta_j e_3
\]
\[
= \eta_i^T C
\begin{pmatrix}
\eta_j &  &  & \\
& \eta_j &  & \\
&  & \eta_j & \\
&  & & \eta_j
\end{pmatrix}
\begin{pmatrix}
e_0 \\
e_1 \\
e_2 \\
e_3
\end{pmatrix},
\]
so this yields the equation (\ref{eq}).
\end{proof}
The proof of the following lemma is straightforward.
\begin{lemma}\label{lem2}
Let $\mathcal{P}$ be a linear transformation on \texorpdfstring{$\bigwedge(\mathbb{R}^2)$}{Λ(R2)}. The following equations
hold:

\begin{align*}
e_i\mathcal{P}(e_j)
	&=
	\eta_j^{T}I_i^{T}
	\begin{pmatrix}
		e_0\\
		e_1\\
		e_2\\
		e_3
	\end{pmatrix},
	\qquad
	\mathcal{P}(e_i)e_1=
	\eta_i^{T}
	\begin{pmatrix}
		0 & 1 & 0 & 0\\
        0 & 0 & 0 & 0\\
		0 & 0 & 0 & -1\\
		0 & 0 & 0 & 0
	\end{pmatrix}
	\begin{pmatrix}
		e_0\\
		e_1\\
		e_2\\
		e_3
	\end{pmatrix},
\end{align*}
\begin{align*}
	\mathcal{P}(e_i)e_2
	&=
	\eta_i^{T}
	\begin{pmatrix}
		0 & 0 & 1 & 0\\
		0 & 0 & 0 & 1\\
		0 & 0 & 0 & 0\\
		0 & 0 & 0 & 0
	\end{pmatrix}
	\begin{pmatrix}
		e_0\\
		e_1\\
		e_2\\
		e_3
	\end{pmatrix},
	\qquad
	\mathcal{P}(e_i)e_3=
	\eta_i^{T}
	\begin{pmatrix}
		0 & 0 & 0 & 1\\
		0 & 0 & 0 & 0\\
		0 & 0 & 0 & 0\\
		0 & 0 & 0 & 0
	\end{pmatrix}
	\begin{pmatrix}
		e_0\\
		e_1\\
		e_2\\
		e_3
	\end{pmatrix},
\end{align*}
where
\begin{align*}
	I_0 &=
	\begin{pmatrix}
		1 & 0 & 0 & 0 \\
		0 & 1 & 0 & 0 \\
		0 & 0 & 1 & 0 \\
		0 & 0 & 0 & 1
	\end{pmatrix}, &
	I_1 &=
	\begin{pmatrix}
		0 & 0 & 0 & 0 \\
		1 & 0 & 0 & 0 \\
		0 & 0 & 0 & 0 \\
		0 & 0 & 1 & 0
	\end{pmatrix}, \\
	I_2 &=
	\begin{pmatrix}
		0 & 0 & 0 & 0 \\
		0 & 0 & 0 & 0\\
		1 & 0 & 0 & 0 \\
		0 & -1 & 0 & 0
	\end{pmatrix}, &
	I_3 &=
	\begin{pmatrix}
		0 & 0 & 0 & 0 \\
		0 & 0 & 0 & 0 \\
		0 & 0 & 0 & 0 \\
		1 & 0 & 0 & 0
	\end{pmatrix}.
\end{align*}
\end{lemma}

We now state our main theorem, which is a consequence of Lemmas \ref{lem1} and \ref{lem2}:

\begin{theorem}
$\mathcal{P}$ is a Rota-Baxter operator of weight $\lambda$ on \texorpdfstring{$\bigwedge(\mathbb{R}^2)$}{Λ(R2)} if and only if the column vectors $\eta_i$ of $P$ and $P$ satisfy the following relations:
\begin{align*}
	\eta_0(4)^{T}C^TP &= P(I_0,I_1,I_2,I_3)\eta_0(4)+P^2+\lambda P,\\
    \eta_1(4)^{T}C^TP &= P(I_0,I_1,I_2,I_3)\eta_1(4)+P
	\begin{pmatrix}
		0&1&0&0\\
		0&0&0&0\\
		0&0&0&-1\\
		0&0&0&0
	\end{pmatrix}^{T}P+\lambda P
	\begin{pmatrix}
		0&1&0&0\\
		0&0&0&0\\
		0&0&0&-1\\
		0&0&0&0
	\end{pmatrix}^{T},\\
\end{align*}
\begin{align*}
	\eta_2(4)^{T}C^TP &= P(I_0,I_1,I_2,I_3)\eta_2(4)+P
			\begin{pmatrix}
				0&0&1&0\\
				0&0&0&1\\
				0&0&0&0\\
				0&0&0&0
			\end{pmatrix}^{T}
			P+\lambda P
			\begin{pmatrix}
				0&0&1&0\\
				0&0&0&1\\
				0&0&0&0\\
				0&0&0&0
			\end{pmatrix}^{T},
			\\
	\eta_3(4)^TC^TP &= P(I_0,I_1,I_2,I_3)\eta_3(4)+P
			\begin{pmatrix}
		    0&0&0&1\\
			0&0&0&0\\
			0&0&0&0\\
            0&0&0&0	
			\end{pmatrix}^{T}
			P+\lambda P
			\begin{pmatrix}
				0&0&0&1\\
				0&0&0&0\\
				0&0&0&0\\
				0&0&0&0
			\end{pmatrix}^{T}.
			 \end{align*}
\end{theorem}
\begin{proof}
	Using Lemma \ref{lem2}, we obtain
	\begin{align*}
	\mathcal{P}(\mathcal{P}(e_i)e_j) = \eta_j^T I_i^T p^T
	\begin{pmatrix}
		e_0 \\
		e_1 \\
		e_2 \\
		e_3
	\end{pmatrix},\\
	\mathcal{P}\bigl(\mathcal{P}(e_i)e_1\bigr)
	&=
	\eta_i^{T}
	\begin{pmatrix}
	0 & 1 & 0 & 0\\
	0 & 0 & 0 & 0\\
	0 & 0 & 0 & -1\\
	0 & 0 & 0 & 0	
	\end{pmatrix}
	P^{T}
	\begin{pmatrix}
		e_0\\
		e_1\\
		e_2\\
		e_3
	\end{pmatrix},
	\\[2ex]
	\mathcal{P}\bigl(\mathcal{P}(e_i)e_2\bigr)
	&=
	\eta_i^{T}
	\begin{pmatrix}
		0 & 0 & 1 & 0\\
	0 & 0 & 0 & 1\\
	0 & 0 & 0 & 0\\
	0 & 0 & 0 & 0
	\end{pmatrix}
	P^{T}
	\begin{pmatrix}
		e_0\\
		e_1\\
		e_2\\
		e_3
	\end{pmatrix},
\end{align*}
\begin{align*}
	\mathcal{P}\bigl(\mathcal{P}(e_i)e_3\bigr)
	&=
	\eta_i^{T}
	\begin{pmatrix}
	0 & 0 & 0 & 1\\
	0 & 0 & 0 & 0\\
	0 & 0 & 0 & 0\\
	0 & 0 & 0 & 0	
	\end{pmatrix}
	P^{T}
	\begin{pmatrix}
		e_0\\
		e_1\\
		e_2\\
		e_3
	\end{pmatrix}.
	\end{align*}		
$\mathcal{P}$ is a Rota-Baxter operator with weight $\lambda$ if and only if for any $i, j = 0, 1, 2, 3$, the following equations holds:
\[
\mathcal{P}(e_i)\mathcal{P}(e_j) = \mathcal{P}\big(\mathcal{P}(e_i)e_j\big) + \mathcal{P}\big(e_i\mathcal{P}(e_j)\big) + \lambda \mathcal{P}(e_i e_j).
\]
For $j = 0$, we obtain from Lemma \ref{lem1} that
\[
\eta_i^T C\eta_0(4) = \eta_0^T I_i^T P^T + \eta_i^T P^T  + \lambda \eta_i^T.
\]
This amounts to the first equality of Theorem.\\
When $j=1$, Lemma \ref{lem1} gives
\begin{equation*}
\eta_0^T C \eta_1(4)= \eta_1^T I_0^T p^T+ \eta_0^T
\begin{pmatrix}
	0&1&0&0\\
    0&0&0&0\\
    0&0&0&-1\\
	0&0&0&0	
\end{pmatrix}
P^{T}
+ \lambda\eta_1^{T},
\end{equation*}
\begin{equation*}
	\eta_1^T C \eta_1(4)= \eta_1^T I_1^T p^T+ \eta_1^T
	\begin{pmatrix}
		0&1&0&0\\
		0&0&0&0\\
		0&0&0&-1\\
		0&0&0&0	
	\end{pmatrix}
	P^{T},
	\end{equation*}
	\begin{equation*}
	\eta_2^T C \eta_1(4)= \eta_1^T I_2^T p^T+ \eta_2^T
	\begin{pmatrix}
		0&1&0&0\\
		0&0&0&0\\
		0&0&0&-1\\
		0&0&0&0	
	\end{pmatrix}
	P^{T}
- \lambda\eta_3^{T},,
	\end{equation*}
    \begin{equation*}	
	   \eta_3^T C \eta_1(4)= \eta_1^T I_3^T p^T+  \eta_3^T
       \begin{pmatrix}
    	0&1&0&0\\
    	0&0&0&0\\
	      0&0&0&-1\\
     	0&0&0&0	
        \end{pmatrix}
         P^{T},
         \end{equation*}
	which is equivalent to the second equality in Theorem 3.1.\\
When $j=2$, \\
\begin{equation*}
	\eta_0^T C \eta_2(4)= \eta_2^T I_0^T p^T+ \eta_0^T
	\begin{pmatrix}
		0&0&1&0\\
		0&0&0&1\\
        0&0&0&0\\
		0&0&0&0	
	\end{pmatrix}
	P^{T}
	+ \lambda\eta_2^{T},
\end{equation*}
	\begin{equation*}
	\eta_1^T C \eta_2(4)= \eta_2^T I_1^T p^T+ \eta_1^T
	\begin{pmatrix}
		0&0&1&0\\
		0&0&0&1\\
		0&0&0&0\\
		0&0&0&0	
	\end{pmatrix}
	P^{T}
	+ \lambda\eta_3^{T},
     \end{equation*}
     \begin{equation*}
     \eta_2^T C \eta_2(4)= \eta_2^T I_2^T p^T+ \eta_2^T
     \begin{pmatrix}
     	0&0&1&0\\
     	0&0&0&1\\
     	0&0&0&0\\
     	0&0&0&0	
     \end{pmatrix}
     P^{T}
     \end{equation*},
   \begin{equation*}
   	\eta_3^T C \eta_2(4)= \eta_2^T I_3^T p^T+ \eta_3^T
   	\begin{pmatrix}
   		0&0&1&0\\
   		0&0&0&1\\
   		0&0&0&0\\
   		0&0&0&0	
   	\end{pmatrix}
   	P^{T},
   \end{equation*}
which is equivalent to the third equality in Theorem 3.1.\\
When $j=3$,\\
   \begin{equation*}
   		\eta_0^T C \eta_3(4)= \eta_3^T I_0^T p^T+ \eta_0^T
   	\begin{pmatrix}
   		0&0&0&1\\
   		0&0&0&0\\
   		0&0&0&0\\
   		0&0&0&0	
   	\end{pmatrix}
   	P^{T}
   	+ \lambda\eta_3^{T},
   \end{equation*}
   \begin{equation*}
   		\eta_1^T C \eta_3(4)= \eta_3^T I_1^T p^T+ \eta_1^T
   	\begin{pmatrix}
   		0&0&0&1\\
   		0&0&0&0\\
   		0&0&0&0\\
   		0&0&0&0	
   	\end{pmatrix}
   	P^{T},
   	\end{equation*}
   	\begin{equation*}
   			\eta_2^T C \eta_3(4)= \eta_3^T I_2^T p^T+ \eta_2^T
   		\begin{pmatrix}
   			0&0&0&1\\
   			0&0&0&0\\
   			0&0&0&0\\
   			0&0&0&0	
   		\end{pmatrix}
   		P^{T},
    \end{equation*}
	\begin{equation*}
   			\eta_3^T C \eta_3(4)= \eta_3^T I_3^T p^T+ \eta_3^T
   		\begin{pmatrix}
   			0&0&0&1\\
   			0&0&0&0\\
   			0&0&0&0\\
   			0&0&0&0	
   		\end{pmatrix}
   		P^{T},		
   	\end{equation*}
which is exactly the fourth equality in Theorem 3.1.
\end{proof}

\section{Rota-Baxter Operators with weight $0$ on \texorpdfstring{$\bigwedge(\mathbb{R}^2)$}{Λ(R2)}}

We begin by determining all $P$ which satisfy the matrix equations in Theorem 3.1. Using this theorem, the matrix system of equations reduces to the following system of equations:
\begin{align}
&& p_{11}^{2} + 2p_{12}p_{21} + 2p_{13}p_{31} + 2p_{14}p_{41} &= 0 \tag{c1} \\
&& 2p_{21}p_{22} + 2p_{23}p_{31} + 2p_{24}p_{41} &= 0 \tag{c2} \\
&& 2p_{21}p_{32} + 2p_{31}p_{33} + 2p_{34}p_{41} &= 0 \tag{c3} \\
&& 2p_{21}p_{42} + 2p_{31}p_{43} + 2p_{41}p_{44} &= 0 \tag{c4} \\
&& p_{11}p_{12} + p_{12}p_{22} + p_{13}p_{32} - p_{14}p_{31} + p_{14}p_{42} &= 0 \tag{c5}\\
&& p_{22}^{2} + p_{23}p_{32} - p_{24}p_{31} + p_{24}p_{42} &= 0 \tag{c6}
\end{align}
\begin{align}
&& p_{22}p_{32} - p_{31}p_{34} + p_{32}p_{33} + p_{34}p_{42} &= 0 \tag{c7} \\
&& -p_{21}p_{32} + p_{22}p_{31} + p_{22}p_{42} - p_{31}p_{44} + p_{32}p_{43} + p_{42}p_{44} &= 0 \tag{c8}\\
&& p_{11}p_{13} + p_{12}p_{23} + p_{13}p_{33} + p_{14}p_{21} + p_{14}p_{43} &= 0 \tag{c9} \\
&& p_{21}p_{24} + p_{22}p_{23} + p_{23}p_{33} + p_{24}p_{43} &= 0 \tag{c10}\\
&& p_{21}p_{34} + p_{23}p_{32} + p_{33}^{2} + p_{34}p_{43} &= 0 \tag{c11} \\
&& -p_{21}p_{33} + p_{21}p_{44} + p_{23}p_{31} + p_{23}p_{42} + p_{33}p_{43} + p_{43}p_{44} &= 0 \tag{c12} \\
&& p_{11}p_{14} + p_{12}p_{24} + p_{13}p_{34} + p_{14}p_{44} &= 0 \tag{c13} \\
&& p_{22}p_{24} + p_{23}p_{34} + p_{24}p_{44} &= 0 \tag{c14} \\
&& p_{24}p_{32} + p_{33}p_{34} + p_{34}p_{44} &= 0 \tag{c15} \\
&& -p_{21}p_{34} + p_{24}p_{31} + p_{24}p_{42} + p_{34}p_{43} + p_{44}^{2} &= 0 \tag{c16} \\
&& p_{11}p_{12} + p_{12}p_{22} + p_{13}p_{32} + p_{14}p_{31} + p_{14}p_{42} &= 0 \tag{c17} \\
&& p_{22}^{2} + p_{23}p_{32} + p_{24}p_{31} + p_{24}p_{42} &= 0 \tag{c18} \\
&& p_{22}p_{32} + p_{31}p_{34} + p_{32}p_{33} + p_{34}p_{42} &= 0 \tag{c19} \\
&& p_{21}p_{32} - p_{22}p_{31} + p_{22}p_{42} + p_{31}p_{44} + p_{32}p_{43} + p_{42}p_{44} &= 0 \tag{c20} \\
&& p_{12}^{2} &= 0 \tag{c21} \\
&& p_{12}p_{13} + p_{14}p_{22} + p_{14}p_{33} &= 0 \tag{c22} \\
&& p_{22}p_{24} + p_{24}p_{33} &= 0 \tag{c23} \\
&& p_{22}p_{34} + p_{33}p_{34} &= 0 \tag{c24}\\
&& -p_{22}p_{33} + p_{22}p_{44} + p_{23}p_{32} + p_{33}p_{44} &= 0 \tag{c25}\\
&& p_{12}p_{14} + p_{14}p_{34} &= 0 \tag{c26} \\
&& p_{34}^{2} &= 0 \tag{c27}\\
&& p_{24}p_{34} &= 0 \tag{c28} \\
&& -p_{22}p_{34} + p_{24}p_{32} + p_{34}p_{44} &= 0 \tag{c29} \\
&& p_{11}p_{13} + p_{12}p_{23} + p_{13}p_{33} - p_{14}p_{21} + p_{14}p_{43} &= 0 \tag{c30} \\
&& -p_{21}p_{24} + p_{22}p_{23} + p_{23}p_{33} + p_{24}p_{43} &= 0 \tag{c31} \\
&& -p_{21}p_{34} + p_{23}p_{32} + p_{33}^{2} + p_{34}p_{43} &= 0 \tag{c32} \\
&& p_{21}p_{33} - p_{21}p_{44} - p_{23}p_{31} + p_{23}p_{42} + p_{33}p_{43} + p_{43}p_{44} &= 0 \tag{c33} \\
&& -p_{12}p_{13} + p_{14}p_{22} + p_{14}p_{33} &= 0 \tag{c34}
\end{align}
\begin{align}
&& p_{22}p_{24} + p_{24}p_{33} &= 0 \tag{c35} \\
&& p_{22}p_{34} + p_{33}p_{34} &= 0 \tag{c36} \\
&& -p_{22}p_{33} + p_{22}p_{44} + p_{23}p_{32} + p_{33}p_{44} &= 0 \tag{c37} \\
&& p_{13}^{2} &= 0 \tag{c38} \\
&& -p_{13}p_{14} + p_{14}p_{24} &= 0 \tag{c39} \\
&& p_{24}^{2} &= 0 \tag{c40} \\
&& p_{24}p_{34} &= 0 \tag{c41} \\
&& p_{23}p_{34} - p_{24}p_{33} + p_{24}p_{44} &= 0 \tag{c42} \\
&& p_{11}p_{14} + p_{12}p_{24} + p_{13}p_{34} + p_{14}p_{44} &= 0 \tag{c43} \\
&& p_{22}p_{24} + p_{23}p_{34} + p_{24}p_{44} &= 0 \tag{c44} \\
&& p_{24}p_{32} + p_{33}p_{34} + p_{34}p_{44} &= 0 \tag{c45} \\
&& p_{21}p_{34} - p_{24}p_{31} + p_{24}p_{42} + p_{34}p_{43} + p_{44}^{2} &= 0 \tag{c46} \\
&& -p_{12}p_{14} + p_{14}p_{34} &= 0 \tag{c47} \\
&& p_{24}p_{34} &= 0 \tag{c48}\\
&& p_{34}^{2} &= 0 \tag{c49} \\
&& -p_{22}p_{34} + p_{24}p_{32} + p_{34}p_{44} &= 0 \tag{c50}\\
&& p_{13}p_{14} + p_{14}p_{24} &= 0 \tag{c51} \\
&& p_{24}^{2} &= 0 \tag{c52} \\
&& p_{24}p_{34} &= 0 \tag{c53}\\
&& p_{23}p_{34} - p_{24}p_{33} + p_{24}p_{44} &= 0 \tag{c54} \\
&& p_{14}^{2} &= 0 \tag{c55}
\end{align}
Since the entries of $P$ are real, equations (c21), (c27), (c38), (c40) and (c55) give
$p_{12}=p_{13}=p_{14}=p_{24}=p_{34}=0$. Substituting these into the remaining equations,
(c5), (c9), (c13)--(c15), (c17), (c22)--(c24), (c26), (c28)--(c30), (c34)--(c36), (c39),
(c41)--(c45) and (c47)--(c54) become trivial, and (c18), (c19), (c31), (c32), (c37), (c46)
coincide with (c6), (c7), (c10), (c11), (c25), (c16), respectively. Equations (c1) and
(c16) now read $p_{11}^2=0$ and $p_{44}^2=0$, so $p_{11}=p_{44}=0$, and $p_{41}$ appears
in no remaining equation, hence it is free. Next, the equations (c6), (c11) and (c25) give
$p_{22}^2=p_{33}^2=-p_{23}p_{32}=-p_{22}p_{33}$, hence $(p_{22}+p_{33})^2=0$ and
$p_{33}=-p_{22}$, so that (c7) and (c10) hold automatically. Taking the sums and
differences of the pairs (c8), (c20) and (c12), (c33), the system is therefore equivalent to
\begin{align*}
&p_{22}^2+p_{23}p_{32}=0, && p_{21}p_{22}+p_{23}p_{31}=0, && p_{21}p_{32}-p_{22}p_{31}=0,\\
&p_{21}p_{42}+p_{31}p_{43}=0, && p_{22}p_{42}+p_{32}p_{43}=0, && p_{23}p_{42}-p_{22}p_{43}=0,
\end{align*}
together with $p_{33}=-p_{22}$. If $p_{23}\neq0$, then
$p_{32}=-p_{22}^2/p_{23}$, $p_{31}=-p_{21}p_{22}/p_{23}$ and $p_{42}=p_{22}p_{43}/p_{23}$,
and the remaining equations hold automatically. If $p_{23}=0$, then $p_{22}=p_{33}=0$; for
$p_{32}\neq0$ this forces $p_{21}=p_{43}=0$, while for $p_{32}=0$ the only remaining
condition is $p_{21}p_{42}+p_{31}p_{43}=0$. Conversely, each of the resulting matrices
satisfies the reduced system, and hence (c1)--(c55). The cases are pairwise disjoint and
exhaust all solutions. This leads to the following main theorem of this section.
\begin{theorem}
$\mathcal{P}$ is a Rota--Baxter operator of weight $0$ on $\bigwedge(\mathbb{R}^2)$ if and
only if the matrix corresponding to $\mathcal{P}$ has one of the following forms:
\begin{center}
$\begin{pmatrix}
0 & 0 & 0 & 0 \\
0 & 0 & 0 & 0 \\
p_{31} & p_{32} & 0 & 0 \\
p_{41} & p_{42} & 0 & 0
\end{pmatrix} \ (p_{32} \neq 0)$,
~~~~~~~~
$\begin{pmatrix}
0 & 0 & 0 & 0 \\
p_{21} & 0 & 0 & 0 \\
p_{31} & 0 & 0 & 0 \\
p_{41} & p_{42} & p_{43} & 0
\end{pmatrix}  \ ( p_{21}p_{42}+p_{31}p_{43} = 0),$
\newline
$\begin{pmatrix}
0 & 0 & 0 & 0 \\
p_{21} & p_{22} & p_{23} & 0 \\
-\frac{p_{21}p_{22}}{p_{23}} & -\frac{p^{2}_{22}}{p_{23}} & -p_{22} & 0 \\
p_{41} & \frac{p_{22}p_{43}}{p_{23}} & p_{43} & 0
\end{pmatrix}, \ ( p_{23} \neq 0)$,
\end{center}
where $p_{21},p_{22},p_{23},p_{31},p_{32},p_{41},p_{42},p_{43}$ are free parameters
(subject to the stated conditions).
\end{theorem}
\section{Rota-Baxter Operators with non-zero weight $\lambda$ on \texorpdfstring{$\bigwedge(\mathbb{R}^2)$}{Λ(R2)}}
We again begin by determining all $P$ which satisfy the matrix equations in Theorem 3.1. Using this theorem, the matrix system of equations reduces to the following system of equations:
\begin{align}
&& p_{11}^{2} + 2p_{12}p_{21} + 2p_{13}p_{31} + 2p_{14}p_{41} + \lambda p_{11} &= 0 \tag{d1} \\
&& 2p_{21}p_{22} + 2p_{23}p_{31} + 2p_{24}p_{41} + \lambda p_{21} &= 0 \tag{d2} \\
&& 2p_{21}p_{32} + 2p_{31}p_{33} + 2p_{34}p_{41} + \lambda p_{31} &= 0 \tag{d3} \\
&& 2p_{21}p_{42} + 2p_{31}p_{43} + 2p_{41}p_{44} + \lambda p_{41} &= 0 \tag{d4} \\
&& p_{11}p_{12} + p_{12}p_{22} + p_{13}p_{32} - p_{14}p_{31} + p_{14}p_{42} + \lambda p_{12} &= 0 \tag{d5}
\end{align}
\begin{align}
&& p_{22}^{2} + p_{23}p_{32} - p_{24}p_{31} + p_{24}p_{42} + \lambda p_{22} &= 0 \tag{d6} \\
&& p_{22}p_{32} - p_{31}p_{34} + p_{32}p_{33} + p_{34}p_{42} + \lambda p_{32} &= 0 \tag{d7} \\
&& -p_{21}p_{32} + p_{22}p_{31} + p_{22}p_{42} - p_{31}p_{44} + p_{32}p_{43} + p_{42}p_{44} + \lambda p_{42} &= 0 \tag{d8} \\
&& p_{11}p_{13} + p_{12}p_{23} + p_{13}p_{33} + p_{14}p_{21} + p_{14}p_{43} + \lambda p_{13} &= 0 \tag{d9} \\
&& p_{21}p_{24} + p_{22}p_{23} + p_{23}p_{33} + p_{24}p_{43} + \lambda p_{23} &= 0 \tag{d10} \\
&& p_{21}p_{34} + p_{23}p_{32} + p_{33}^{2} + p_{34}p_{43} + \lambda p_{33} &= 0 \tag{d11} \\
&& -p_{21}p_{33} + p_{21}p_{44} + p_{23}p_{31} + p_{23}p_{42} + p_{33}p_{43} + p_{43}p_{44} + \lambda p_{43} &= 0 \tag{d12} \\
&& p_{11}p_{14} + p_{12}p_{24} + p_{13}p_{34} + p_{14}p_{44} + \lambda p_{14} &= 0 \tag{d13} \\
&& p_{22}p_{24} + p_{23}p_{34} + p_{24}p_{44} + \lambda p_{24} &= 0 \tag{d14} \\
&& p_{24}p_{32} + p_{33}p_{34} + p_{34}p_{44} + \lambda p_{34} &= 0 \tag{d15}\\
&& -p_{21}p_{34} + p_{24}p_{31} + p_{24}p_{42} + p_{34}p_{43} + p_{44}^{2} + \lambda p_{44} &= 0 \tag{d16} \\
&& p_{11}p_{12} + p_{12}p_{22} + p_{13}p_{32} + p_{14}p_{31} + p_{14}p_{42} + \lambda p_{12} &= 0 \tag{d17} \\
&& p_{22}^{2} + p_{23}p_{32} + p_{24}p_{31} + p_{24}p_{42} + \lambda p_{22} &= 0 \tag{d18} \\
&& p_{22}p_{32} + p_{31}p_{34} + p_{32}p_{33} + p_{34}p_{42} + \lambda p_{32} &= 0 \tag{d19} \\
&& p_{21}p_{32} - p_{22}p_{31} + p_{22}p_{42} + p_{31}p_{44} + p_{32}p_{43} + p_{42}p_{44} + \lambda p_{42} &= 0 \tag{d20}\\
&& p_{12}^{2} &= 0 \tag{d21} \\
&& p_{12}p_{13} + p_{14}p_{22} + p_{14}p_{33} + \lambda p_{14} &= 0 \tag{d22} \\
&& p_{22}p_{24} + p_{24}p_{33} + \lambda p_{24} &= 0 \tag{d23} \\
&& p_{22}p_{34} + p_{33}p_{34} + \lambda p_{34} &= 0 \tag{d24} \\
&& -p_{22}p_{33} + p_{22}p_{44} + p_{23}p_{32} + p_{33}p_{44} + \lambda p_{44} &= 0 \tag{d25}\\
&& p_{12}p_{14} + p_{14}p_{34} &= 0 \tag{d26} \\
&& p_{34}^{2} &= 0 \tag{d27}\\
&& p_{24}p_{34} &= 0 \tag{d28} \\
&& -p_{22}p_{34} + p_{24}p_{32} + p_{34}p_{44} &= 0 \tag{d29} \\
&& p_{11}p_{13} + p_{12}p_{23} + p_{13}p_{33} - p_{14}p_{21} + p_{14}p_{43} + \lambda p_{13} &= 0 \tag{d30} \\
&& -p_{21}p_{24} + p_{22}p_{23} + p_{23}p_{33} + p_{24}p_{43} + \lambda p_{23} &= 0 \tag{d31} \\
&& -p_{21}p_{34} + p_{23}p_{32} + p_{33}^{2} + p_{34}p_{43} + \lambda p_{33} &= 0 \tag{d32} \\
&& p_{21}p_{33} - p_{21}p_{44} - p_{23}p_{31} + p_{23}p_{42} + p_{33}p_{43} + p_{43}p_{44} + \lambda p_{43} &= 0 \tag{d33}
\end{align}
\begin{align}
&& -p_{12}p_{13} + p_{14}p_{22} + p_{14}p_{33} + \lambda p_{14} &= 0 \tag{d34} \\
&& p_{22}p_{24} + p_{24}p_{33} + \lambda p_{24} &= 0 \tag{d35} \\
&& p_{22}p_{34} + p_{33}p_{34} + \lambda p_{34} &= 0 \tag{d36} \\
&& -p_{22}p_{33} + p_{22}p_{44} + p_{23}p_{32} + p_{33}p_{44} + \lambda p_{44} &= 0 \tag{d37} \\
&& p_{13}^{2} &= 0 \tag{d38} \\
&& -p_{13}p_{14} + p_{14}p_{24} &= 0 \tag{d39} \\
&& p_{24}^{2} &= 0 \tag{d40} \\
&& p_{24}p_{34} &= 0 \tag{d41} \\
&& p_{23}p_{34} - p_{24}p_{33} + p_{24}p_{44} &= 0 \tag{d42} \\
&& p_{11}p_{14} + p_{12}p_{24} + p_{13}p_{34} + p_{14}p_{44} + \lambda p_{14} &= 0 \tag{d43}\\
&& p_{22}p_{24} + p_{23}p_{34} + p_{24}p_{44} + \lambda p_{24} &= 0 \tag{d44} \\
&& p_{24}p_{32} + p_{33}p_{34} + p_{34}p_{44} + \lambda p_{34} &= 0 \tag{d45} \\
&& p_{21}p_{34} - p_{24}p_{31} + p_{24}p_{42} + p_{34}p_{43} + p_{44}^{2} + \lambda p_{44} &= 0 \tag{d46} \\
&& -p_{12}p_{14} + p_{14}p_{34} &= 0 \tag{d47} \\
&& p_{24}p_{34} &= 0 \tag{d48}\\
&& p_{34}^{2} &= 0 \tag{d49} \\
&& -p_{22}p_{34} + p_{24}p_{32} + p_{34}p_{44} &= 0 \tag{d50} \\
&& p_{13}p_{14} + p_{14}p_{24} &= 0 \tag{d51} \\
&& p_{24}^{2} &= 0 \tag{d52} \\
&& p_{24}p_{34} &= 0 \tag{d53}\\
&& p_{23}p_{34} - p_{24}p_{33} + p_{24}p_{44} &= 0 \tag{d54} \\
&& p_{14}^{2} &= 0. \tag{d55}
\end{align}
Since the entries of $P$ are real, equations (d21), (d27), (d38), (d40) and (d55) give $p_{12}=p_{13}=p_{14}=p_{24}=p_{34}=0$. Substituting these into the remaining equations, (d5), (d9), (d13)--(d15), (d17), (d22)--(d24), (d26), (d28)--(d30), (d34)--(d36), (d39), (d41)--(d45) and (d47)--(d54) become trivial, and (d18), (d19), (d31), (d32), (d37), (d46) coincide with (d6), (d7), (d10), (d11), (d25), (d16), respectively. Taking the sums and differences of the pairs (d8), (d20) and (d12), (d33) and writing
\[
B=\begin{pmatrix}p_{22}&p_{23}\\ p_{32}&p_{33}\end{pmatrix},\quad
v=\begin{pmatrix}p_{21}\\ p_{31}\end{pmatrix},\quad
z=\begin{pmatrix}p_{42}\\ p_{43}\end{pmatrix},\quad
s=p_{22}+p_{33}+\lambda,
\]
the system is equivalent to
\begin{align*}
&p_{11}(p_{11}+\lambda)=0,\qquad p_{44}(p_{44}+\lambda)=0,\qquad p_{23}\,s=p_{32}\,s=0,\\
&p_{22}^2+p_{23}p_{32}+\lambda p_{22}=0,\qquad p_{33}^2+p_{23}p_{32}+\lambda p_{33}=0,\qquad p_{23}p_{32}-p_{22}p_{33}+p_{44}\,s=0,\\
&(2B+\lambda I)v=0,\qquad (p_{22}-p_{44})p_{31}=p_{21}p_{32},\qquad (p_{33}-p_{44})p_{21}=p_{23}p_{31},\\
&\bigl(B^{T}+(p_{44}+\lambda)I\bigr)z=0,\qquad 2\,(v\cdot z)+(2p_{44}+\lambda)\,p_{41}=0.
\end{align*}
In particular, $p_{11}$ appears only in the first equation, so $p_{11}\in\{0,-\lambda\}$ independently of the other entries, and likewise $p_{44}\in\{0,-\lambda\}$.

Eliminating $p_{23}p_{32}$ from the quadratic equations by means of the third equation of the second line shows that they are equivalent to $(p_{22}-p_{44})s=(p_{33}-p_{44})s=0$. Together with $p_{23}s=p_{32}s=0$, this yields a dichotomy. Either $B=p_{44}I$, or $s=0$.

If $B=p_{44}I$, then $s=2p_{44}+\lambda=\pm\lambda\neq0$, and the linear equations force $v=0$, $z=0$ and $p_{41}=0$. This gives the diagonal matrices (1) and (2) in the next theorem, according to $p_{44}=0$ or $p_{44}=-\lambda$.

If $s=0$, then $\operatorname{tr}B=-\lambda$ and $\det B=p_{44}s=0$, so $B$ has the distinct eigenvalues $0$ and $-\lambda$. The matrix $2B+\lambda I$ then has eigenvalues $\pm\lambda$ and is invertible. Hence $v=0$, that is $p_{21}=p_{31}=0$, and since $2p_{44}+\lambda=\pm\lambda\neq 0$ we also get $p_{41}=0$. The only remaining conditions are
\[
p_{22}+p_{33}=-\lambda,\qquad p_{23}p_{32}=p_{22}p_{33},
\]
together with

\[
p_{33}p_{42}=p_{32}p_{43},\quad p_{23}p_{42}=p_{22}p_{43}\quad (\text{for } p_{44}=0),
\]
\[
p_{22}p_{42}+p_{32}p_{43}=0,\quad p_{23}p_{42}+p_{33}p_{43}=0\quad (\text{for } p_{44}=-\lambda).
\]
Solving these constraints according to whether $p_{23}\neq0$ or $p_{23}=0$ (in which case $p_{22}(p_{22}+\lambda)=0$, that is $p_{22}\in\{0,-\lambda\}$) gives the matrices (3)--(5) for $p_{44}=0$ and (6)--(8) for $p_{44}=-\lambda$. 
Conversely, each of these matrices satisfies the reduced system, and hence (d1)--(d55). The cases are pairwise disjoint and exhaust all solutions. This leads to the following main theorem of this section.
\begin{theorem}
$\mathcal{P}$ is a Rota-Baxter operator of weight $\lambda\neq 0$ on {$\bigwedge(\mathbb{R}^2)$} if and only if the matrix corresponding to $\mathcal{P}$ has one of the following forms, where $p_{11}\in\{0,-\lambda\}$ in every case:\\
$\begin{pmatrix}
p_{11}&0&0&0\\ 0&0&0&0\\ 0&0&0&0\\ 0&0&0&0
\end{pmatrix}$....(1),~~~~~
$\begin{pmatrix}
p_{11}&0&0&0\\ 0&-\lambda&0&0\\ 0&0&-\lambda&0\\ 0&0&0&-\lambda
\end{pmatrix}$....(2),\\

$\begin{pmatrix}
p_{11}&0&0&0\\ 0&p_{22}&p_{23}&0\\ 0&-\dfrac{p_{22}(\lambda+p_{22})}{p_{23}}&-\lambda-p_{22}&0\\
0&\dfrac{p_{22}p_{43}}{p_{23}}&p_{43}&0
\end{pmatrix}$,....(3)~(for $p_{23}\neq 0$),~~~
$\begin{pmatrix}
p_{11}&0&0&0\\ 0&0&0&0\\ 0&p_{32}&-\lambda&0\\ 0&-\dfrac{p_{32}p_{43}}{\lambda}&p_{43}&0
\end{pmatrix}$....(4),

$\begin{pmatrix}
p_{11}&0&0&0\\ 0&-\lambda&0&0\\ 0&p_{32}&0&0\\ 0&p_{42}&0&0
\end{pmatrix}$....(5),~~~~~
$\begin{pmatrix}
p_{11}&0&0&0\\ 0&p_{22}&p_{23}&0\\ 0&-\dfrac{p_{22}(\lambda+p_{22})}{p_{23}}&-\lambda-p_{22}&0\\
0&\dfrac{(\lambda+p_{22})p_{43}}{p_{23}}&p_{43}&-\lambda
\end{pmatrix}$,....(6)~(for $p_{23}\neq 0$),

$\begin{pmatrix}
p_{11}&0&0&0\\ 0&0&0&0\\ 0&p_{32}&-\lambda&0\\ 0&p_{42}&0&-\lambda
\end{pmatrix}$,....(7),~~~~~
$\begin{pmatrix}
p_{11}&0&0&0\\ 0&-\lambda&0&0\\ 0&p_{32}&0&0\\ 0&\dfrac{p_{32}p_{43}}{\lambda}&p_{43}&-\lambda
\end{pmatrix}$....(8).\\

Here all unspecified parameters ($p_{22},p_{23},p_{32},p_{42},p_{43}$ as they appear) are arbitrary real numbers, subject only to the stated conditions.
\end{theorem}

\vspace{1cm}
%






\end{document}